\documentclass[a4paper,12pt]{article}

\usepackage{amsmath, amssymb, amscd, amsthm}
\usepackage{mathrsfs}
\usepackage[font=small,skip=0pt]{caption}

\usepackage{csquotes, color}
\usepackage{colortbl}
\usepackage[numbers]{natbib}
\usepackage{verbatim}
\usepackage[toc,page]{appendix}
\usepackage{caption}
\usepackage{textcomp}
\usepackage[inline]{enumitem}
\usepackage{amssymb}
\usepackage{tcolorbox}
\usepackage{cancel}
\usepackage{relsize}
\usepackage{bm}

\usepackage{authblk}

\usepackage{algorithm}
\usepackage{algorithmic}
\usepackage{scrextend}
\addtokomafont{labelinglabel}{\sffamily}

\usepackage{tikz}
\usetikzlibrary{arrows.meta}
\usetikzlibrary{positioning}
\usetikzlibrary{shadows}
\usetikzlibrary{calc, intersections}
\usetikzlibrary{shapes.multipart}
\usetikzlibrary{patterns}

\newtheorem{theorem}{Theorem}
\newtheorem{corollary}{Corollary}[theorem]
\newtheorem{notation}{Notation}
\newtheorem{definition}{Definition}
\newtheorem*{definition*}{Definition}

\newtheorem*{example*}{Example}

\newtheorem*{conditional*}{Conditional}

\newtheorem*{remark*}{Remark}
\newtheorem{property}{Property}[definition]

\tikzset{Myarrow/.style={very thin, arrows=-Latex}}

\definecolor{Gray}{gray}{0.85}
\definecolor{LightCyan}{rgb}{0.88,1,1}

\newcolumntype{a}{>{\columncolor{Gray}}c}
\newcolumntype{b}{>{\columncolor{Gray}}r}
\newcolumntype{d}{>{\columncolor{white}}c}

\newcommand{\cA}{{\cal A}}

\newcommand{\cC}{{\cal C}}

\newcommand{\cL}{{\cal L}}

\newcommand{\cS}{{\cal S}}

\newcommand{\cX}{{\cal X}}

\newcommand{\scrB}{{\mathscr B}} 
\newcommand{\scrD}{{\mathscr D}} 
 
\newcommand{\scrG}{{\mathscr G}} 
\newcommand{\scrL}{{\mathscr L}} 
\newcommand{\scrP}{{\mathscr P}}

\usepackage{hyperref}
\hypersetup{
    colorlinks=true,
    citecolor=black,
    linkcolor=black,
    filecolor=cyan,      
    urlcolor=blue,
}
 
\newenvironment{myitemize}
{ \begin{itemize}
    \setlength{\itemsep}{0pt}
    \setlength{\parskip}{0pt}
    \setlength{\parsep}{0pt}     }
{ \end{itemize}                  } 

\newenvironment{myenumerate}
{ \begin{enumerate}
    \setlength{\itemsep}{0pt}
    \setlength{\parskip}{0pt}
    \setlength{\parsep}{0pt}     }
{ \end{enumerate}                  }

\providecommand{\keywords}[1]
{
  \small	
  \textbf{\textit{Keywords---}} #1
}

\title{
A Cluster Model for Growth of Random Trees 
}
\author{Nomvelo Karabo Sibisi}
\affil{University of Cape Town \\ {\small {\tt sbsnom005@myuct.ac.za}}}
\date{June 2021}

\begin{document}
\maketitle
\thispagestyle{empty}

\begin{abstract}
\noindent
We first consider  the growth of trees by  probabilistic attachment  of new vertices to leaves.
This leads rather naturally  to a growth model based on vertex clusters and probabilities assigned to clusters.
This  model turns out to be readily applicable to  attachment at any  depth of the tree,
hence the paper evolves to a general study of  tree growth by cluster-based attachment. 
Drawing inspiration from the concept of intrinsic vertex fitness  due to Bianconi and Barab\'{a}si~\cite{BianconiBarabasi1},
we introduce vertex mass as an {\it additive} intrinsic vertex attribute.
Unlike Bianconi and Barab\'{a}si who used fitness as a  vertex degree multiplier
in the context of growth by preferential attachment,
 we treat vertex mass as a fundamental probabilistic  construct whose additivity plays a primary role.
Notably, independent  mass distributions induce a distribution on the sum of such masses  
through Laplace convolution.
In this way, clusters of vertices inherit their mass distributions from  vertices within the cluster.

Our main contribution is a novel theorem 
 for the joint distribution of  cluster masses,
conditioned on their respective   distributions.
As described by Ferguson~\cite{Ferguson1} and Kingman~\cite{Kingman} 
in the context of distributions on general measures, 
the choice of gamma conditioning distributions leads to the Dirichlet distribution.
The degree-based distribution arises as the mean of the Dirichlet distribution,
with preferential attachment  based on this mean.
The fitness scheme of~\cite{BianconiBarabasi1} and affine preferential attachment can be understood in this light.
The latter can also be related to random forests, 
with creation of new trees in addition to attachment to an existing tree. 

Beyond  gamma conditioning  distributions, 
our  theorem allows other choices, such as the fat-tailed stable distributions with infinite mean.
We discuss  L\'{e}vy conditioning distributions as a gamma alternative,
the L\'{e}vy distribution being a notable instance  of the stable family.
We conclude with a theorem giving the analytic marginals of the normalised distribution
conditioned on the L\'{e}vy distribution.


\end{abstract}

\section{Introduction}
\label{sec:intro}
In the first instance,  we consider a tree that grows by attachment from the root 
(where a directed edge from one vertex to another represents attachment of the former to the latter).
The tree has two  vertex types:\\
(i)  a {\it deep} vertex  with one or more incoming attachments \\
(ii) a {\it leaf}   (or shallow vertex)
with no incoming attachments (the root is also taken to be a leaf at the start). \\
We consider growth by attachment of a new vertex to a leaf and not to a deep vertex.
Figure~\ref{fig:treegraph} is an illustrative sequence of `snapshots' of a tree as it  grows by leaf attachment 
(time increases  to the right in each snapshot).
The rules for growth by leaf attachment are as follows:
\vspace{-1pc}
 \begin{labeling}{Rule 5:}
\setlength{\partopsep}{0pt}
\setlength{\topsep}{0pt}
\setlength{\itemsep}{0pt}
\setlength{\parskip}{0pt}
\setlength{\parsep}{0pt} 
\item[Rule~1:] Only new vertices may issue attachments, existing 
 vertices are `inactive'.
\item[Rule~2:] For a tree, a new vertex  issues exactly one attachment 
\item[Rule~3:] {\it Leaf attachment}: a new vertex may only attach to a leaf.
 \item[Rule~4:] A leaf can receive multiple attachments from different  new vertices created in the same time step.
 Thereafter, the leaf becomes deep and may not receive further attachments at later times.
 \item[Rule~5:] New vertices select attachment targets from a probability distribution over leaves.
\end{labeling}
\vspace{-1pc}

 \begin{figure}[tbh]
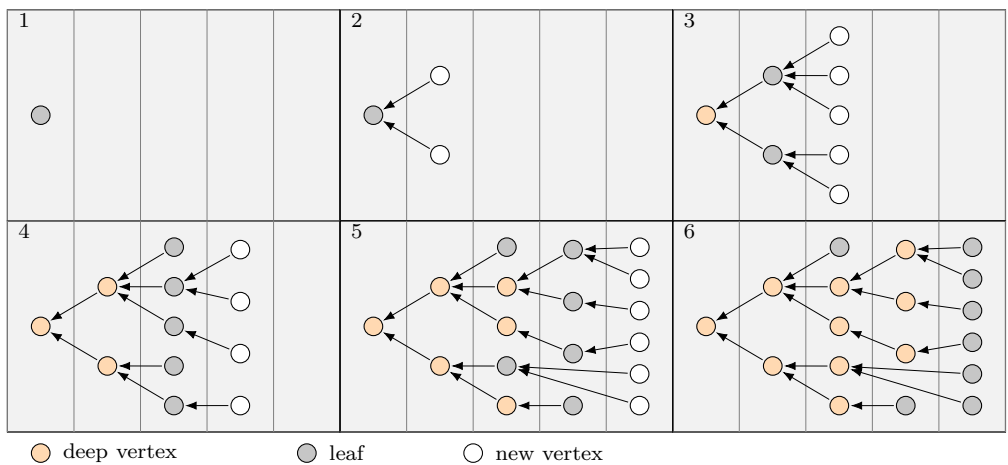

\centering
\include{Figure_tree_6frames} 
\caption{
Tree grows sequentially from  1 to 6 through leaf attachment. 
}
\label{fig:treegraph}
\end{figure}


There are two logically distinct steps at play. 
The first step is the  creation of $n_i\ge0$ new vertices in the $i^{\rm th}$ time interval.
We  may take such vertex creation  to follow a Poisson process, 
{\it i.e.}\ at each time step $i$, the number $n_i$ is drawn from a Poisson 
distribution with a given rate. 
The second step is the attachment of  newly created vertices  to leaves of the  existing tree. 
It is the latter step that is of primary interest -- 
choosing a leaf probability distribution from which new vertices select their  attachment targets.

If we ignore the history of the tree altogether, then there is no reason to prefer one leaf over another.
Hence the leaf distribution to be used for attachment may be taken to be uniform.
However, we do wish  to  strike a probabilistic distinction amongst leaves 
based on how well-connected they are to the history of the tree.
Broadly speaking, we want to encourage branching of the tree
 or, equivalently,  clustering of vertices and discourage 
 isolated chains of vertices.
A local view of  clustering  is to consider  attachments that leaves make to their immediate predecessors.
The more leaves cluster by attaching to a common predecessor, the more they are to be favoured for attachment. 
Hence a cluster, {\it i.e.} a  set of vertices that directly attach  to a common vertex, is the object of primary interest.
It thus makes sense to 
consider a probability distribution over such clusters of leaves rather than treating the leaves as isolated objects.
Once a  cluster is selected, a leaf within the cluster may then be selected on the basis that leaves within a cluster are  equivalent, 
{\it i.e.}\ the leaf distribution is uniform within a cluster. 

 \begin{figure}[tbh]
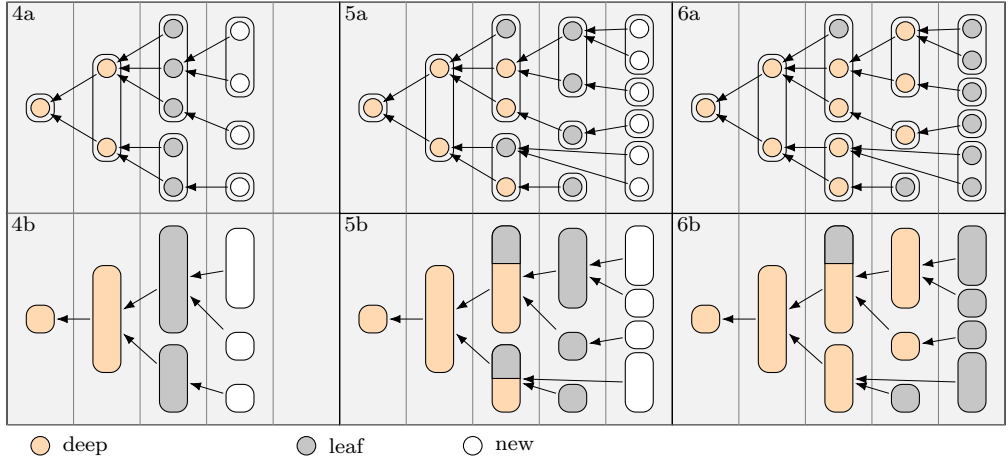

\centering
\include{Figure_tree_cells} 
\caption{
Ringed vertices are  clusters of leaves attaching to a common deep vertex.
}
\label{fig:treecells}
\end{figure}
We may now reinterpret the example of Figure~\ref{fig:treegraph} in terms of clusters and attachments amongst  them, 
as illustrated in Figure~\ref{fig:treecells}.
There are  clusters of deep vertices, clusters of leaves and hybrid clusters containing both.

An immediate  observation at this point is that  the cluster interpretation need not be restricted to leaf attachment.
It is also applicable to growth by unrestricted  attachment to an existing vertex at any depth of the tree.
We shall refer to the latter as free attachment.
In either case, for a given vertex $v$, let a cluster $\cC_v$ be the set of of vertices that attach to $v$
($v$ itself belongs to a cluster that attaches  to some earlier vertex, except for the root).
If $v$ is a leaf then $\cC_v$ is the empty set.
Let  $|\cC_v|$ be a quantitative property  of $\cC_v$. 
How we define $|\cC_v|$ is what is at the core of this paper, more so than whether we are dealing with leaf or free attachment.
An obvious choice of $|\cC_v|$ is the number of vertices in  $\cC_v$, or the in-degree of $v$. 
We shall elaborate below on the more nuanced choice of $|\cC_v|$ taken here.
The fundamental   difference between leaf and free  attachment lies in how we use $\cC_v$. 

{\bf Leaf Attachment}
\vspace{-1pc}
 \begin{labeling}[$\bullet$]{$\bullet$}
 \setlength{\partopsep}{0pt}
\setlength{\topsep}{0pt}
\setlength{\itemsep}{0pt}
\setlength{\parskip}{0pt}
\setlength{\parsep}{0pt} 
\item A cluster $\cC_v$  is static, containing a fixed number of vertices  in a single time interval that attach to $v$. 
\item For any deep vertex $v$ for which $\cC_v$ contains one or more leaves, a new vertex first selects  $\cC_v$
and then selects a  leaf within $\cC_v$   to attach to. 
A cluster with no leaves is at full capacity and is thus not available for selection.
\item The probabilistic  selection  of $\cC_v$ is based on  its $|\cC_v|$ relative to other clusters containing at least one  leaf.
\end{labeling}
{\bf Free Attachment}
\vspace{-1pc}
 \begin{labeling}[$\bullet$]{$\bullet$}
 \setlength{\partopsep}{0pt}
\setlength{\topsep}{0pt}
\setlength{\itemsep}{0pt}
\setlength{\parskip}{0pt}
\setlength{\parsep}{0pt} 
\item A cluster $\cC_v$ is dynamic with no simple topology --  $\cC_v$ evolves  as  new vertices attach to $v$.  
\item  A new vertex directly selects an existing vertex $v$ at any depth to attach to. 
Hence an existing vertex can receive an unlimited number of attachments.
\item The probabilistic  selection  of $v$ is based on its  $|\cC_v|$ relative to all  existing vertices.
If  $|\cC_v|$ is the in-degree of $v$ then leaves would be excluded since they would have $|\cC_v|=0$.
A simple remedy is to introduce a global offset $\beta>0$ so that $|\cC_v|\to|\cC_v|+\beta$.
\end{labeling}

The  conceptual benefit of  the leaf attachment problem is that it has prompted us  
to think in terms of clusters and  attributes we might assign to clusters that are more general than vertex in-degree.
Such thinking does not arise naturally in the free attachment case where clusters are topologically awkward.
We have thus chosen to retain  the thought process flowing from  leaf attachment
even though the probabilistic formulation of the paper is ultimately not restricted to leaf attachment.

Attachment based on  vertex in-degree 
is  known as preferential attachment.
It is widely used in network science (Barab\'{a}si~\cite{Barabasi}, van der Hofstad~\cite{vdHofstad}, Newman~\cite{Newman},
Coscia~\cite{Coscia}).
Price~\cite{Price} referred to such attachment
as cumulative advantage in his graphical model of citation networks (a generalisation of a tree to a directed graph where a 
new vertex may make multiple attachments to   existing vertices).
Preferential attachment is  often referred to as a \textquote{rich get richer} scheme 
that rewards vertices that are already rich in attachments.

To construct a cluster distribution, we look beyond mere vertex count to consider an  intrinsic vertex attribute 
that we refer to as vertex mass, which is positive and additive. 
Hence the mass of a cluster is a sum of the masses of the vertices in the cluster.
The probabilistic model may thus be summarised as follows.
 We take vertex  mass to be 
 governed by an  assigned probability distribution.
 We shall take the vertices to be independent. 
 The mass of a cluster is thus 
 the sum of the masses of the vertices in the cluster,
 whose distribution is  induced by the distributions of the masses of the vertices on the cluster.
Specifically, the distribution of a sum of independently distributed variables is a convolution of the distributions of the individual variables.
This is a crucial consistency property  that underpins our mass-based model -- 
the mass distribution of a cluster is a convolution of the mass distributions of its constituent vertices.

This may beg the question of whether we can simply invent  intrinsic attributes  at will. 
We take the view that the justification for any model lies in its ability to generate specified behaviour,
or reproduce observed behaviour if it is intended to be a model of the external world. 
It suffices  that the model  exhibit internal mathematical consistency which, in our case, is the probabilistic convolution structure 
imposed by  additivity.
For variables on the real line $(-\infty,\infty)$, such convolution is the familiar Fourier convolution.
Restriction to the nonnegative half-line $[0,\infty)$ gives  Laplace convolution, which will be the workhorse of our model.

We do not claim originality for the notion of intrinsic vertex attributes.
Intrinsic vertex mass is inspired by the concept of  intrinsic vertex fitness due to  Bianconi and  Barab\'{a}si~\cite{BianconiBarabasi1}.  
 The  focus of~\cite{BianconiBarabasi1} was growth of a graph by the degree-based scheme of preferential attachment.
 Hence the  fitness distribution was used solely to generate a single sample per vertex  that is in turn used
as a degree  multiplier to tune the degree distribution.
Beyond the inspiration, we adopt a different modelling route from~\cite{BianconiBarabasi1} based on vertex mass.
In fact, it will emerge that a degree multiplier can arise in a rather different context for our probabilistic model.

The primary objective of 
preferential attachment  is to explore asymptotic behaviour -- 
whether or not the limiting degree distribution obeys scale-free (power law) behaviour  
(Barab\'{a}si and Albert~\cite{BarabasiAlbert}). 
With fitness included, it is  possible to generate  \textquote{winner takes all} degree clustering reminiscent of
Bose-Einstein condensation in physics (Bianconi and  Barab\'{a}si~\cite{BianconiBarabasi2}). 
In pursuit  of  scale-free behaviour by an  alternative route to preferential attachment,
Caldarelli {\it et al.}~\cite{Caldarelli} explored the fitness distribution in its own right instead of the degree distribution. 

Point processes involve similar probabilistic constructs to those considered here.
The conceptual difference is that  
while our clusters may visually resemble spatial cells, such clusters only come into being as a result of  vertex attachment, 
they do not pre-exist like some partition of a spatial domain may do even before any random points are strewn across it.
Nonetheless, graphs can  usefully be modelled as  point processes.
In an approach inspired by the adjacency matrix, Caron and Fox~\cite{CaronFox} model a graph as a point process on $\mathbb{R}_+^2$
whose points are pairs of connected vertices.
 Each vertex has an associated positive sociability parameter which, in turn, is a point of a Poisson point process
 or a jump of a  completely random measure.
The intrinsic  sociability of a vertex  rather than its degree is the fundamental probabilistic attribute. 
The approach has much in common with other work on random measures
({\it e.g.} Ferguson~\cite{Ferguson1, Ferguson2}, Sibisi and Skilling~\cite{SibisiSkilling}).

\paragraph{Main Result:}
At the outset, we had the rather focussed objective of constructing  a growth model for the leaf attachment problem.
As our thinking evolved, it became clearer that we needed first to consider a problem of much broader scope.

Accordingly, our primary contribution is significantly  more far-reaching than the paper's initial brief.
It takes the form of Theorem~\ref{thm:jointprob}, which gives  
the joint distribution of  cluster masses, conditioned on  their respective independent  distributions.
A sample from this distribution is  itself a  distribution that, in the original application context, 
we  may use for attachment in our tree growth application.

To our awareness, Theorem~\ref{thm:jointprob} is novel, at least to the extent that it accommodates any 
set of conditioning distributions.
The idea of such generality is contained in Kingman~\cite{Kingman}, despite the difference in approach.
Choosing gamma conditioning distributions leads to the Dirichlet distribution,  initially  described by Ferguson~\cite{Ferguson1}.
He addressed the more general problem of a Dirichlet process, which may be described  as a set of consistent Dirichlet distributions 
over different   partitions of an interval or  spatial domain. 

We shall see  that  the mean of the Dirichlet distribution is the degree distribution.
Hence we may interpret preferential attachment as probabilistic selection scheme  based on the Dirichlet mean.
However, a representative sample from the Dirichlet distribution can be very different  from the mean, depending on the 
parameters of the conditioning gamma distributions.

Furthermore, the generality of Theorem~\ref{thm:jointprob} allows the choice of any conditioning distributions. 
 An example is  the family of fat-tailed stable distributions with infinite mean, such as  the L\'{e}vy distribution.
We conclude with a theorem giving the analytic form of the marginals of the normalised distribution
conditioned on the L\'{e}vy distribution.
 This distribution is a novel alternative to the Dirichlet distribution induced by gamma conditioning distributions.
 We defer more detailed study to a separate  paper.

Toward a more detailed discussion of the model, we start with some well-known  preliminaries.

\section{Preliminaries}
\label{sec:prelim}
We restrict attention to distributions  defined on the nonnegative half-line $[0,\infty)$ 
and  take every distribution to be normalised and to have a density. 
\begin{notation}
A probability distribution $F(x)$  and its density $f(x)\ge0$ with respect to $dx$ are related as follows
\begin{equation}
F(x)  = \int_0^x dF(u)= \int_0^x f(u) du  
\end{equation}
 with normalisation $F(\infty)=1$.
\end{notation}

\begin{definition}
The Laplace transform of  a probability distribution $F(x)$  is
\begin{align}
\cL\{F\}(s) \equiv \widetilde{F}(s) &= \int_0^\infty e^{-sx} dF(x) \qquad s\ge0
\label{eq:laplaceF}
\intertext{
If   $F(x)$ has a  density $f(x)$, $\widetilde{F}(s)$ may be  written as $\tilde{f}(s)$}
\cL\{f\}(s) \equiv \tilde{f}(s) &= \int_0^\infty e^{-sx} f(x) dx  
\label{eq:laplacef}
\end{align}
\end{definition}
A  unit jump at $x_0$ in $F$ corresponds to an atom  at $ x_0$ in  $f(x)$ represented by the  
Dirac delta $\delta(x-x_0)$. 

\begin{definition}
The Laplace convolution $f_1\star f_2$ of  two functions  is defined by 
\begin{align}
(f_1 \star f_2)(x) &= \int_0^x f_1(u) f_2(x-u) du =  \int_0^x f_1(x-u) f_2(u) du 
\label{eq:convolution}
\end{align}
\end{definition}
Convolution is associative: $f_1\star f_2\star f_3 = f_1\star (f_2\star f_3) = (f_1\star f_2)\star f_3$, {\it etc}. 
Hence the definition readily generalises to an arbitrary number of functions.
Following Feller, we use the notation  $f^{n\star}$ for the $n$-fold self-convolution of $f$.

Theorem~\ref{thm:lplcconv} is standard and so is the proof.
\begin{tcolorbox}
\begin{theorem} 
The Laplace transform of a convolution of functions on $[0,\infty)$ is the product of the Laplace transforms of the individual functions
\begin{align}
\cL\{f_1\star\dots\star f_n\}(s) &=   \tilde{f_1}(s)\times\dots\times\tilde{f_n}(s)
\label{eq:convtheorem} \\
\cL\{f^{n\star}\}(s) &=  \left( \tilde{f}(s)\right)^n
\label{eq:selfconvtheorem}
\end{align}
\label{thm:lplcconv}
\end{theorem}
\vspace{-2.5pc}
\end{tcolorbox}

\begin{proof}[Proof of Theorem \ref{thm:lplcconv}]
Consider the Laplace transform of the  convolution of 2 functions 
\begin{align*}
\cL\{f_1\star f_2\}(s) 
&=  \int_0^\infty e^{-sx} \left[\int_0^x f_1(u)f_2(x-u) du\right] dx \\
&=  \int_0^\infty f_1(u) \left[\int_u^\infty e^{-sx}  f_2(x-u)dx\right] du \\
&=  \int_0^\infty f_1(u) \left[\int_0^\infty e^{-s(u+y)}  f_2(y)dy\right] du \\
&=  \int_0^\infty e^{-su} f_1(u)du \int_0^\infty e^{-sy}  f_2(y)dy \\
&= \tilde{f_1}(s) \tilde{f_2}(s)
\end{align*}
The generalisation~(\ref{eq:convtheorem}) and  particular case~(\ref{eq:selfconvtheorem}) follow  from associativity of convolution.
\end{proof}

Theorem~\ref{thm:sumconv} is standard.
\begin{tcolorbox}
\begin{theorem} 
Let $\{x_1,\dots,x_n\}\in[0,\infty)$ be independently distributed
with distributions $\{F_1,\ldots,F_n\}$ and associated  densities $\{f_1,\ldots, f_n\}$  respectively.
The   distribution of the sum $z=x_1+\dots+x_n$ 
is  the Laplace convolution of the $n$ individual distributions, \\
{\it i.e.}\ the  density of  $z$ is $\Pr(z|F_1,\ldots,F_n)\equiv\Pr(z|f_1,\ldots,f_n) = (f_1\star \dots\star f_n)(z)$. 
\label{thm:sumconv}
\end{theorem}
\tcblower
{\bf Note.} To limit notational clutter, we shall often write  $\Pr(z|f_1,\ldots,f_n)$ merely as $\Pr(z)$, 
where the conditioning information can be inferred from the context.
\end{tcolorbox}

\begin{proof}[Proof of Theorem \ref{thm:sumconv}]
Consider first $n=2$. 
$\Pr(x_1,x_2)=\Pr(x_1)\Pr(x_2)=f_1(x_1)f_2(x_2)$ since $x_1, x_2$ are independent. 
Also, $z=x_1+x_2$ so that $\Pr(z|x_1,x_2)=\delta(z-x_1-x_2)$.
Then, starting from the joint density $\Pr(z,x_1,x_2)$, we may marginalise to obtain $\Pr(z)\equiv\Pr(z|f_1,f_2)$: 
\begin{alignat}{3}
             &\Pr(z,x_1,x_2) &&= \Pr(z|x_1,x_2)\Pr(x_1,x_2)          &&= \delta(z-(x_1+x_2))f_1(x_1)f_2(x_2) 
\label{eq:przx1x2} \\
\implies & \Pr(z,x_1)       &&= \int_0^\infty \Pr(z,x_1,x_2) dx_2 &&= f_1(x_1) \int_0^\infty  \delta(x_2-(z-x_1))f_2(x_2)dx_2 \nonumber \\
             &                       &&                                                        &&= f_1(x_1) f_2(z-x_1)   
\label{eq:przx1} \\
\implies &  \Pr(z)             &&= \int_0^\infty \Pr(z,x_1) dx_1        &&= \int_0^z f_1(x_1)f_2(z-x_1) dx_1 \\
&&&&&= (f_1\star f_2)(z)
\label{eq:prz}
\end{alignat}
By associativity,  $z=(x_1+x_2)+x_3 \implies \Pr(z)=((f_1\star f_2)\star f_3)(z)=(f_1\star f_2\star f_3)(z)$.
The general result follows: $z=x_1+x_2+\dots+x_n \implies \Pr(z|f_1,\ldots,f_n) =(f_1\star \dots\star f_n)(z)$.
\end{proof}

\section{Main Theorem}
\label{sec:maintheorem}

As in Theorem~\ref{thm:sumconv}, 
let $\{x_1,\dots,x_n\}\in[0,\infty)$ be independently distributed
with distributions $\{F_1,\ldots,F_n\}$ (densities $\{f_1,\ldots, f_n\}$)  respectively
and let  $z=x_1+\dots+x_n$.
Define  normalised variables $\{p_i : x_i=zp_i, i=1,\ldots, n\}$ 
 so that  $\{p_1,\dots,p _n\}\in[0,1]$ with $ p_1+\dots+p_n=1$.
 Then $(p_1,\ldots,p_n)$ may be looked upon as a probability distribution $P(n)$ of an $n$-valued discrete variable: 
$\{P(i|n) = p_i :  i=1,\ldots, n\}$.
Hence a probability distribution of $(p_1,\ldots,p_n)$ may be regarded as a probability distribution of 
a  probability distribution of a  discrete variable.

To our awareness,  Theorem~\ref{thm:jointprob} 
is novel, at least as stated in  general form with any $\{f_1,\ldots,f_n\}$.
\begin{tcolorbox}
\begin{theorem} 
The multivariate   probability distribution  of the  probability distribution $P(n)=(p_1,\ldots,p_n)$, 
conditioned on  $\{f_1,\ldots, f_n\}$, has density: 
\begin{align}
\Pr(p_1,\ldots,p_{n-1}|f_1,\ldots,f_n) 
        &= \int_0^\infty z^{n-1}\prod_{i=1}^n f_i(zp_i)\, dz 
        = \int_0^\infty \frac{dz}{z} \prod_{i=1}^n zf_i(zp_i)  \nonumber 
\end{align}
where  
$p_n=1- (p_1+\dots+p_{n-1})$. 
\label{thm:jointprob}
\end{theorem}

\begin{corollary} 
Let $f_{(i)}$ be the convolution of $(f_1,\ldots,f_n)$  with $f_i$ omitted.
The  $n-1$ marginal distributions  have densities
\begin{align}
\Pr(p_i|f_1,\ldots,f_n) 
  &\equiv \int_0^1 \Pr(p_1,\dots,p_{n-1}|f_1,\ldots,f_n)\,  dp_1\dots dp_{i-1}dp_{i+1}\dots dp_{n-1} \nonumber \\
    &= \int_0^\infty zf_i(zp_i)\,f_{(i)}(z(1-p_i))\, dz \qquad  i=1\ldots n-1  \nonumber
\end{align}
\label{cor:marginal}
\end{corollary}
\vspace{-2pc}
\end{tcolorbox}

\begin{proof}[Proof of Theorem \ref{thm:jointprob}]
With implicit conditioning on $\{f_1,\ldots, f_n\}$, we have 
$\Pr(z|x_1,\dots,x_n)=\delta(z-\sum_{i=1}^n x_i)$ and 
$\Pr(x_1,\ldots,x_n)=\prod_{i=1}^nf_i(x_i)$,
so that 
\begin{alignat}{2}
&\Pr(z, x_1,\ldots,x_n) 
  &&= \Pr(z|x_1,\dots,x_n) \Pr(x_1,\ldots,x_n)   \nonumber \\
&&&= \delta(z-(x_1+\dots+x_n)) \prod_{i=1}^n f_i(x_i) \nonumber \\
\implies\quad & \Pr(z, x_1,\ldots,x_{n-1})
  &&\equiv  \int_0^\infty \Pr(z,x_1,\ldots,x_n) dx_n  \nonumber \\
&&&=  \int_0^\infty  \delta(z-(x_1+\dots+x_n)) \prod_{i=1}^n f_i(x_i)\,dx_n \nonumber \\
&&&= \prod_{i=1}^n f_i(x_i)  \qquad x_n=z-(x_1+\dots+x_{n-1}) 
\label{eq:jointprobz}
\end{alignat}
With $\{zp_i=x_i\}$, 
$\Pr(z, x_1,\ldots,x_{n-1})dx_1\ldots dx_{n-1} = \Pr(z,zp_1,\ldots,zp_{n-1}) z^{n-1}dp_1\ldots dp_{n-1}$ 
\begin{alignat}{3}
\implies & \Pr(z, p_1,\ldots,p_{n-1}) &&= z^{n-1}\Pr(z,zp_1,\ldots,zp_{n-1})
&&=z^{n-1}\prod_{i=1}^n f_i(zp_i)  = \frac{1}{z} \prod_{i=1}^n z f_i(zp_i) \nonumber \\
\implies & \Pr(p_1,\ldots,p_{n-1})   &&=  \int_0^\infty  \frac{dz}{z} \prod_{i=1}^n z f_i(zp_i)  && p_n=1-(p_1+\dots+p_{n-1})
\label{eq:jointprob}
\end{alignat}
where $\Pr(p_1,\ldots,p_{n-1})\equiv \Pr(p_1,\ldots,p_{n-1}|f_1,\ldots,f_n)$.
\end{proof}

\begin{proof}[Proof of Corollary \ref{cor:marginal}]
$f_{(i)}$ is the convolution of $(f_1,\ldots,f_n)$  with $f_i$ omitted.
By Theorem~\ref{thm:sumconv}:
\begin{align}
x_{(i)} &= \sum_{j\ne i} ^n x_j 
\implies \Pr(x_{(i)}) = f_{(i)}(x_{(i)})  
\end{align}
Also, $z=x_i+x_{(i)}$ so that $\Pr(z|x_i, x_{(i)})=\delta(z-(x_i+x_{(i)})$. 
Hence, similar to the $n=2$ case of Theorem~\ref{thm:sumconv}:
\begin{alignat}{2}
              &\Pr(z, x_i,x_{(i)}) &&= \Pr(z|x_i, x_{(i)})\Pr(x_i,x_{(i)})  \nonumber \\
                                       &  &&= \delta(z-(x_i+x_{(i)})\Pr(x_i)\Pr(x_{(i)}) \\
\implies  &\Pr(z,x_i)            &&= f_i(x_i)\int_0^\infty \delta(x_{(i)}-(z-x_i))\  f_{(i)}(x_{(i)}) dx_{(i)} \nonumber \\
              &                           &&= f_i(x_i) f_{(i)}(z-x_i)  \qquad  i=1\ldots n-1 
\label{eq:marginalx} \\
\implies  &\Pr(z,p_i)            &&= z f_i(zp_i) f_{(i)}(z(1-p_i)) \qquad  zp_i=x_i \\
\implies  &\Pr(p_i)               &&= \int_0^\infty z f_i(zp_i) f_{(i)}(z(1-p_i))\, dz
\label{eq:marginal}
\end{alignat}
where $\Pr(p_i)\equiv \Pr(p_i|f_1,\ldots,f_n)$.
\end{proof}

Theorem~\ref{thm:jointprob} is valid for any set of conditioning distributions $\{F_1,\ldots,F_n\}$
with respective densities $\{f_1,\ldots,f_n\}$.
It places no additional restriction on the  distributions,  such as the existence of means.
Indeed, in Theorem~\ref{thm:levymarginal} below,  we will consider conditioning distributions with infinite means.
We have implicitly  assumed finite $n$. 
Whether~(\ref{eq:jointprob}) and (\ref{eq:marginal}) exist for $n\to\infty$ will depend on $\{f_1,\ldots,f_n\}$.

Corollary~\ref{cor:marginal} enables evaluation of marginals 
without the need first to evaluate  the joint distribution and then marginalise explicitly. 

Adopting the convention of representing a distribution by its conditioning information,  
let $\scrP(F_1,\ldots,F_n)$ or $\scrP(f_1,\ldots,f_n)$ denote the distribution with density~(\ref{eq:jointprob}) 
of Theorem~\ref{thm:jointprob}.
By the foregoing discussion, to draw a sample from $\scrP(F_1,\ldots,F_n)$:
\begin{tcolorbox}
 \begin{labeling}{$\bullet$}
 \setlength{\partopsep}{0pt}
\setlength{\topsep}{0pt}
\setlength{\itemsep}{0pt}
\setlength{\parskip}{0pt}
\setlength{\parsep}{0pt} 
\item[$\bullet$] Draw $n$ independent  samples $\{\hat{x}_i \sim F_i:  i=1\ldots n\}$ from the $n$ distributions $\{F_i\}$  
\item[$\bullet$] Form $\hat{z}=\hat{x}_1+\dots+\hat{x}_n$ and $\{\hat{p}_i=\hat{x}_i/\hat{z}:  i=1\ldots n\}$
\item[$\bullet$] Then $\hat{P}(n)=(\hat{p}_1,\ldots,\hat{p}_n)$ is a  sample point  from   $\scrP(F_1,\ldots,F_n)$. 
$\hat{P}(n)$ is a normalised distribution on an $n$-valued  discrete variable that may then be used as appropriate
(such as selecting the $i^{\rm th}$ cluster given  $n$  clusters).
\end{labeling}
\end{tcolorbox}


The key observation here is that we do not actually need $\scrP(F_1,\ldots,F_n)$ (or its marginals)
to generate a sample from it.
It suffices to generate independent samples from the conditioning distributions  $\{F_1,\ldots,F_n\}$  and normalise.
The benefit of Theorem~\ref{thm:jointprob} is that it shows how formally to obtain the density of $\scrP(F_1,\ldots,F_n)$ (and its marginals).
If this can be done  in closed form, then it facilitates an analytical  study of the  properties of  $\scrP(F_1,\ldots,F_n)$.
It also enables visual representation of $\scrP(F_1,\ldots,F_n)$, at least for small $n$, while 
the $n$ marginals can be  visualised individually  in one-dimension for any $n$.
In the absence of a closed form, the exercise becomes computational rather than analytical --  
we have to rely on many samples from $\scrP(F_1,\ldots,F_n)$ to approximate its properties,
but without a theoretical handle on how many samples are needed for a \textquote{good} approximation.

It is logical then, at least in the first instance, to search for a set   of conditioning distributions $\{F_1,\ldots,F_n\}$
for which Theorem~\ref{thm:jointprob} gives a closed form of the density of $\scrP(F_1,\ldots,F_n)$.
To that end, we turn  to the set of gamma conditioning distributions, for which it will emerge that
Theorem~\ref{thm:jointprob} leads to the Dirichlet distribution, 
as introduced by Ferguson~\cite{Ferguson1}   in his seminal work on the Dirichlet process.

\section{Gamma Conditioning  Distributions}
\label{sec:gamma}
The gamma distribution $\scrG(\alpha,\lambda)$ has density
\begin{align}
\Pr(x|\alpha,\lambda) \equiv \gamma_{\alpha,\lambda}(x) &= \frac{\lambda^{\alpha}}{\Gamma(\alpha)} x^{\alpha-1} e^{-\lambda x}
\label{eq:gamma} 
\end{align}
where $\Gamma(\alpha)$ is the gamma function  $\int_0^\infty x^{\alpha-1} e^{-x}dx$ ($\Gamma(n+1)=n!$ for integer $n$), 
$\alpha>0$ is the shape parameter and  $\lambda>0$ is the decay rate.
The  shape parameter plays a fundamental role: 
\begin{myenumerate}
\item $\alpha=1$:  the density $\gamma_{1,\lambda}=\lambda e^{-\lambda x}$, a simple  exponential.
\item $\alpha>1$:  $\gamma_{\alpha,\lambda}(0)=0$ and $\gamma_{\alpha,\lambda}(x)$ peaks and decays to infinity rather like a skewed Gaussian. 
\item $\alpha<1$: $\gamma_{\alpha,\lambda}(x)$ has \textquote{$1/x$} behaviour, concentrating its mass toward $x=0$.
\end{myenumerate}
For integer shape parameter, $\scrG(n,\lambda)$ is also known as the Erlang distribution,
named after Erlang who discovered  it in his pioneering study of  telephone networks.

The Laplace transform of~{\rm(\ref{eq:gamma})}  is
\begin{align}
\widetilde{\gamma}_{\alpha,\lambda}(s) &= \left(\frac{\lambda}{\lambda+s}\right)^\alpha
\label{eq:lplcgamma} \\
\text{with mean}\quad 
-\widetilde{\gamma}^\prime_{\alpha,\lambda}(0) &= \frac{\alpha}{\lambda} 
\label{eq:gammamean}
\end{align}
It is evident from~(\ref{eq:lplcgamma})  that 
$\widetilde{\gamma}_{\alpha_1,\lambda}(s)\times\widetilde{\gamma}_{\alpha_2,\lambda}(s)
=\widetilde{\gamma}_{\alpha_1+\alpha_2,\lambda}(s)$.
Hence the gamma distribution satisfies  the following closure property:
\begin{property}[Gamma closure]
For a  given decay rate,  the gamma distribution is closed under convolution with the shapes combining additively:
 $\scrG(\alpha_1,\lambda)\star\scrG(\alpha_2,\lambda)= \scrG(\alpha_1+\alpha_2,\lambda)$.
\label{prop:gammaclosure}
\end{property}

\subsection{Beta Distribution}
\label{sec:beta}
Turning to Theorem~\ref{thm:jointprob}, 
let $f_1(x_1)=\gamma_{\alpha_1,\lambda}(x_1)$ and $f_2(x_2)=\gamma_{\alpha_2,\lambda}(x_2)$.
Then, for $z=x_1+x_2$, $ \Pr(z|f_1,f_2) = (\gamma_{\alpha_1,\lambda}\star \gamma_{\alpha_2,\lambda})(z)
=\gamma_{\alpha_1+\alpha_2,\lambda}(z)$.
Hence~(\ref{eq:jointprob}) gives 
\begin{align}
\Pr(p|f_1,f_2) 
&= \int_0^\infty z\,\gamma_{\alpha_1,\lambda}(zp)\,\gamma_{\alpha_2,\lambda}(z(1-p))\, dz\\
&= \int_0^\infty \left[z\frac{\lambda^{\alpha_1}}{\Gamma(\alpha_1)}(zp)^{\alpha_1-1}e^{-\lambda zp}\right]
                \left[z\frac{\lambda^{\alpha_2}}{\Gamma(\alpha_2)}(z(1-p))^{\alpha_2-1}e^{-\lambda z(1-p)}\right] \frac{dz}{z}  \nonumber \\
&=   \left[\lambda^{\alpha_1+\alpha_2}\int_0^\infty z^{\alpha_1+\alpha_2-1} e^{-\lambda z} dz \right] 
                \frac{p^{\alpha_1-1} (1-p)^{\alpha_2-1}}{\Gamma(\alpha_1)\Gamma(\alpha_2)}  \nonumber \\              
&=  \frac{\Gamma(\alpha_1+\alpha_2)}{\Gamma(\alpha_1)\Gamma(\alpha_2)}\, p^{\alpha_1-1} (1-p)^{\alpha_2-1} \qquad  p\in[0,1]
\label{eq:beta}
\end{align}
The common decay $\lambda$ has cancelled out of $\Pr(p|f_1,f_2)$.
Hence, in the notation introduced above,  we may write the corresponding distribution $\scrP(f_1,f_2)$ simply as $\scrP(\alpha_1,\alpha_2)$, 
where $\alpha_1,\alpha_2$ are understood to be shape parameters of 2 gamma distributions with any common decay $\lambda>0$.
Correspondingly, we may write $\Pr(p|f_1,f_2)$ as $\Pr(p|\alpha_1,\alpha_2)$, 
the density of what is known as  the beta distribution  $\scrB(\alpha_1,\alpha_2)$.
Using the recursion $\Gamma(\alpha+1)=\alpha\Gamma(\alpha)$, the mean of $\scrB(\alpha_1,\alpha_2)$ is
\begin{align}
\mu\equiv\mathbb{E}[p] = 
\int_0^1 p \Pr(p|\alpha_1,\alpha_2) dp 
    &= \frac{\alpha_1}{\alpha_1+\alpha_2} \int_0^1 \Pr(p|\alpha_1+1,\alpha_2) dp \nonumber \\
    &= \frac{\alpha_1}{\alpha_1+\alpha_2} 
\label{eq:betamean} 
\end{align}

\subsection{Dirichlet Distribution}
\label{sec:dirichlet}
Generalising beyond $\scrB(\alpha_1,\alpha_2)$, 
consider  gamma densities $\{f_i(x_i)=\gamma_{\alpha_i,\lambda}(x_i): i=1,\ldots,n\}$
and $z=x_1+\dots+x_n$.
Letting $\alpha=\alpha_1+\dots+\alpha_n$,
$ \Pr(z|f_1,\ldots,f_n) = (\gamma_{\alpha,\lambda})(z)$.
Hence, with  $p_n=1-(p_1+\ldots+p_{n-1})$,~(\ref{eq:jointprob}) gives 
\begin{align}
\Pr(p_1,\ldots,p_{n-1}|\alpha_1,\ldots,\alpha_n) 
&= \int_0^\infty \frac{dz}{z} \prod_{i=1}^n z \,\gamma_{\alpha_i,\lambda}(zp_i)  \\
&= \Gamma(\alpha)\prod_{i=1}^n \frac{p_i^{\alpha_i-1}}{\Gamma(\alpha_i)}  \qquad  p_i\in[0,1]
\end{align}
This is the density of the multivariate beta distribution, which is more commonly  known as the Dirichlet distribution 
$\scrD(\alpha_1,\ldots,\alpha_n)$ 
with $n$ shape parameters arising from the $n$ conditioning  gamma distributions sharing a common decay $\lambda$.
There is no loss of generality in choosing $\lambda=1$. 
The derivation of the Dirichlet  distribution from $\scrG(\alpha_i,1)$ dates back to Ferguson~\cite{Ferguson1}.
However, we have derived the Dirichlet distribution as a particular case of Theorem~(\ref{thm:jointprob}), 
which admits any choice of conditioning distributions $\{F_i\}$.

\subsection{Dirichlet Marginals}
\label{sec:margnals}
Recalling Corollary~\ref{cor:marginal}, $f_{(i)}=f_1\star\dots\star f_n$ with $f_i$ omitted.
For $f_i=\gamma_{\alpha_i,\lambda}$  it readily follows that $f_{(i)}=\gamma_{\alpha-\alpha_i,\lambda}$.
Hence, by~(\ref{eq:marginal}) of Corollary~\ref{cor:marginal},
 the marginal distributions of $\scrD(\alpha_1,\ldots,\alpha_n)$ have densities 
\begin{align}
\Pr(p_i|\alpha_1,\ldots,\alpha_n) 
               &=  \frac{\Gamma(\alpha)}{\Gamma(\alpha_i)\Gamma(\alpha-\alpha_i)}\, p_i^{\alpha_i-1} (1-p_i)^{(\alpha-\alpha_i)-1}  
               \qquad (i=1\ldots n-1) 
\label{eq:betamarginal}
\end{align}
This is the density of the beta distribution $\scrB(\alpha_i,\alpha-\alpha_i)$, whose mean is
\begin{align}
\mu_i\equiv\mathbb{E}[p_i]   &= \frac{\alpha_i}{\alpha}  \qquad (i=1\ldots n-1) 
\label{eq:marginalmean}
\end{align}
and $\mu_n$ is defined by  $\mu_n= 1-(\mu_1+\dots+\mu_{n-1})=\alpha_n/\alpha$.

\subsection{Dirichlet Properties}
\label{sec:properties}
The Dirichlet distribution inherits its behaviour  from its conditioning gamma distributions.
Hence the shape parameters have a crucial role, as discussed above for the gamma distribution.
Shape   parameters below 1 push the probability mass toward the axes, whilst tending to concentrate it away from the axes for 
shape parameters above 1.
If all shape parameters  are 1 then the distribution is uniform.
For any $n$, if we choose   $\{\alpha_i>1\}$ for all $i$,  then we induce a $\scrD(\alpha_1,\ldots,\alpha_n)$ 
with probability concentrated away from the axes.
The larger such $\{\alpha_i\}$ are, the more sharply concentrated around its mean $\scrD(\alpha_1,\ldots,\alpha_n)$ will be.
On the other hand, if $\{\alpha_i<1\}$ for some or all $i$, then then we induce $\scrD(\alpha_1,\ldots,\alpha_n)$ 
where some or all its probability is pushed to the axes, although the mean may still lie away from the axes.
In the limit where we allow $\{\alpha_i=0\}$ for some but not all $i$, then a sample $\hat{P}(n)=(\hat{p}_1,\ldots,\hat{p}_n)$
from $\scrD(\alpha_1,\ldots,\alpha_n)$ will have $\hat{p}_i=0$ for corresponding $\alpha_i=0$, so that 
$\hat{P}(n)$ is concentrated on those $\hat{p}_i$ corresponding to $\alpha_i>0$.
An equivalent way of stating this is that the density of the conditioning distribution $\scrG(\alpha_i=0,1)$ is an atom at zero.
Yet the Dirichlet mean $\bar{P}(n)$ may remain quite smooth.

Visual representation would be of benefit but multivariate distributions are tricky to visualise.
The beta-distributed marginals $\scrB(\alpha_i,\alpha-\alpha_i)$ are a useful way to get  
a sense of the behaviour of the Dirichlet distribution.
It is also common to represent the Dirichlet distribution on an $n-1$ dimensional simplex of $n$-dimensional space 
but this is not easy to visualise beyond $n=3$.
 
The Dirichlet distribution often arises as a building block for a Dirichlet process -- 
a  family  of consistent Dirichlet distributions for different $n$.
The typical context involves progressive refinement of some spatial domain (often motivated through the metaphor  of 
repeated breaking of pieces of a stick).
In his paper, Ferguson~\cite{Ferguson1}  introduced
arbitrary partitioning  of a  general space $\cX$: 
$\cX= \cA_1\cup\dots\cup\cA_n$  ($\cA_i\cap\cA_j=\emptyset, i\ne j$). 
He then defined  a measure $\alpha$ on $\cX$: $\alpha_i\equiv\alpha(\cA_i)$ is the \textquote{size} of $\cA_i$.
Then defining a distribution  $\scrG(\alpha_i,1)$ on another measure $F_i\equiv F(\cA_i)$ and normalising using
$F(\cX)$ gives a Dirichlet distribution $\scrD(\alpha_1,\ldots,\alpha_n)$ on any $n$-cell partition of $\cX$.
Ferguson thus defined the Dirichlet process as a consistent family of such Dirichlet distributions for different $n$, 
 which makes sense because  the measure  $\alpha$ is additive under cell combination (union of disjoint sets)
 and the conditioning  gamma distributions are closed under addition of shape parameters.
As Ferguson established, in the $n\to\infty$  limit of an arbitrarily fine partition, the cell sizes approach zero 
so that  samples from  $\scrD(\alpha_1,\ldots,\alpha_n)$ will necessarily be  
concentrated on isolated cells,  rather as described above for $\alpha_i\ll 1$. 

With the general probabilistic framework in place, we may turn to a discussion of our model.

\section{Probabilistic Model}
\label{sec:probmodel}

Let there be $n$  clusters of interest, where  cluster $i$ has  $k_i$ vertices
(a cluster of interest will be determined by whether we are interested in  leaf attachment or deep attachment,
but   that distinction is not of particular significance at this point).
Let all  vertices  have  independent  identically distributed  masses with distribution $\scrG(\eta,1)$ 
for a chosen shape parameter $\eta>0$.
Hence, 
cluster $i$ has   mass distribution $\scrG(\eta k_i,1)$.
The corresponding normalised mass distribution is the Dirichlet distribution $\scrD(\eta k_1,\ldots, \eta k_n)$.

The next step is to generate a normalised  sample $\hat{P}(n)=(\hat{p}_1,\ldots,\hat{p}_n)\sim\scrD(\eta k_1,\ldots, \eta k_n)$
as discussed in Section~\ref{sec:maintheorem}.
$\hat{P}(n)$ is then a representative distribution over the $n$ clusters that may be used for attachment.
As discussed earlier, representative samples from $\scrD(\eta k_1,\ldots, \eta k_n)$  behave radically differently depending on whether
$\{\eta k_i<1\}$ or $\{\eta k_i>1\}$.
In the former case, a small number of clusters can dominate the distribution of a representative distribution $\hat{P}(n)$ because of the    concentration of  probability near the axes.
Thus, a small number of clusters (or vertices to which they attach) can attract the lion's share of the next wave of  attachments. 


How does this compare with common practice for growth by attachment? 
The popular approach of preferential attachment may be seen as a special case of what has been presented  here.

\subsection{Preferential Attachment}
\label{sec:prefattach}
The mean of $\scrD(\eta k_1,\ldots, \eta k_n)$ is the discrete distribution $\bar{P}(n)=(\mu_1,\ldots\mu_n)$
where 
\begin{align}
\mu_i &= \frac{k_i}{\sum_j k_j} \qquad  i=1\ldots n 
\label{eq:prefattach}
\end{align}
As previously noted, attachment based on such vertex in-degree (vertex count of attachments to a given vertex) is known as 
preferential attachment.
The mean~(\ref{eq:prefattach})  is insensitive to $\eta$. 
Hence the flexibility  that can be induced by a global $\eta$ on a representative Dirichlet sample  $\hat{P}$ is not visible to 
the mean $\bar{P}$ used in preferential attachment.

\subsubsection{Bianconi-Barab\'{a}si Fitness Scheme}
As  a variant of the foregoing, let cluster $i$ have an assigned intrinsic shape parameter $\eta_i$ 
and let each vertex within the cluster have a mass distribution $\scrG(\eta_i,1)$.
This leads to  $\scrD(\eta_1 k_1,\ldots, \eta_n k_n)$ and means 
\begin{align}
\mu_i &= \frac{\eta_i k_i}{\sum_j \eta_jk_j} \qquad  i=1\ldots n 
\label{eq:preffitness}
\end{align}
This is the   scheme of Bianconi and Barab\'{a}si~\cite{BianconiBarabasi1}, where the degree multipliers $\{\eta_i\}$
are referred to as fitness parameters.
We do not speak to the choice of the  $\{\eta_i\}$, which  could be sampled from an assigned distribution on $\eta$.

\subsubsection{Affine Preferential Attachment}
Let the   intrinsic vertex distribution $\scrG(\eta,1)$ be complemented 
by  an intrinsic  cluster distribution $\scrG(\beta,1)$ assigned to all clusters.
The effective  distribution of cluster $i$ is $\scrG(\eta k_i,1)\star\scrG(\beta,1) = \scrG(\eta k_i+\beta,1)$.
Alternatively, we can simply say that each cluster's shape parameter is displaced by $\beta$. 
This induces  $\scrD(\eta k_1+\beta,\ldots, \eta k_n+\beta)$, with means
\begin{align}
\mu_i &= \frac{\eta k_i+\beta}{\sum_j \eta k_j+\beta}   \qquad  i=1\ldots n  
\label{eq:prefattach2}
\end{align}
Using~(\ref{eq:prefattach2}) for attachment is known as affine preferential attachment in the random graph literature.
It  has been used to model the growth of random trees 
({\it e.g.}\ Garavaglia~{\it et al.}~\cite{GaravagliaHofstad}, 
Marchand and Manolescu~\cite{Marchand}).

\subsubsection{Random Forest}
For completeness, we note that the functional form of affine preferential attachment can be interpreted and used differently
by writing it as:  
\begin{align}
\mu_i &= \frac{\eta k_i+\beta}{\sum_j \eta k_j+\beta}   
          \equiv \frac{\eta k_i+\delta+(\beta-\delta)}{\sum_j \eta k_j+\beta}   
\label{eq:forest}
\end{align}
For $0<\delta<\beta$,~(\ref{eq:forest}) may be used for the two distinct steps:
\vspace{-1pc}
 \begin{labeling}{Attachment:}
 \setlength{\partopsep}{0pt}
\setlength{\topsep}{0pt}
\setlength{\itemsep}{0pt}
\setlength{\parskip}{0pt}
\setlength{\parsep}{0pt} 
\item[Attachment:] Free attachment to vertex $i$  with probability $\propto \eta k_i+\delta$ 
ensures that a leaf $l$, which has  in-degree $k_l=0$,  can also receive direct attachments with probability $\propto \delta$.
\item[Creation:]  Creation of a new root, unattached to any previous vertex,
  with  probability $\propto \beta-\delta$. 
\end{labeling}
\vspace{-1pc}
Creation of new roots, or planting of random  trees,  in addition to attachment to existing vertices
results in a random forest even though the process starts from a single rooted tree.

\subsubsection{Chinese Restaurant Metaphor}
The random forest interpretation above is conceptually similar to the
Chinese Restaurant metaphor, 
where a new guest joins an occupied table  $i$ with probability $\propto k_i$ (number of guests  already seated at table $i$), 
or starts a new table with probability  $\propto \beta$.
However, there is no need in this case  to introduce $\delta$ because  tables are isolated  and not linked like vertices.
Hence the concept of a deep table and a shallow table (leaf) does not arise, $k_i$ is table occupancy rather than in-degree.

The whole of the foregoing discussion arises from  gamma conditioning distributions sharing a common decay, taken to be 1.
But Theorem~\ref{thm:jointprob} does not impose  gamma conditioning distributions at all.
We explore  an alternative choice next.

\section{Stable Conditioning Distributions}
\label{sec:stabletrees}

Consider a distribution $\cS(\alpha,\nu)$  for $\alpha>0$ and  $0<\nu<1$, 
with density $f_{\alpha,\nu}(x)$ whose Laplace transform is
\begin{align}
\tilde f_{\alpha,\nu}(s) &= \exp \left(-\alpha s^\nu \right) 
\label{eq:stablelplc} \\
\implies -\tilde f^\prime_{\alpha,\nu}(s) &=\frac{\nu}{s^{1-\nu}} \, \tilde f_{\alpha,\nu}(s)
\label{eq:stablelplcprime}
\end{align}
Distributions with Laplace  transform~(\ref{eq:stablelplc}) are referred  as \textquote{stable}, for reasons we shall not dwell on for present purposes.
It suffices here to note that: 
\vspace{-1pc}
 \begin{labeling}{Stable 2:}
 \setlength{\partopsep}{0pt}
\setlength{\topsep}{0pt}
\setlength{\itemsep}{0pt}
\setlength{\parskip}{0pt}
\setlength{\parsep}{0pt} 
\item[Stable 1:] $\cS(\alpha,\nu)$ is fat-tailed with infinite mean: 
$0<\nu<1 \implies -\tilde f^\prime_{\alpha,\nu}(0)=\infty$.
\item[Stable 2:] $\cS(\alpha_1,\nu)\star\cS(\alpha_2,\nu) = \cS(\alpha_1+\alpha_2,\nu)$. 
$\cS(\alpha,\nu)$ is closed under convolution for fixed $\nu$, like the gamma distribution.
This  follows readily from the Laplace  transform~(\ref{eq:stablelplc}).
\end{labeling}


Th challenge  is that analytic forms of the density $f_{\alpha,\nu}(x)$ are elusive.
A well-known instance is the case $\nu=1/2$, 
known as the L\'{e}vy distribution $\scrL(\alpha)\equiv\cS(\alpha,\frac{1}{2})$, with density 
\begin{align}
 f_{\alpha,\frac{1}{2}}(x) &= \frac{\alpha\,e^{-\alpha^2/4x}}{2\sqrt{\pi x^3}} 
\label{eq:levy}
\end{align}
(We are aware of only one other analytic form for $\nu=1/3$ that we will not discuss here.)

The  L\'{e}vy distribution is fat-tailed, with infinite mean like all stable distribution on $(0,\infty)$.
It has power-law asymptotic behaviour $f_{\alpha,\frac{1}{2}}(x) \sim \alpha x^{-3/2}$. 
Theorem~\ref{thm:jointprob} states that, for $\scrL_i\equiv\scrL(\alpha_i)$ and any $n$,
we can construct the multivariate distribution $\scrP(\scrL_1,\ldots,\scrL_n)$. 
We present here the analytic  form of the marginals rather than the  full joint density.

To our awareness, Theorem~\ref{thm:levymarginal} is novel.
\begin{tcolorbox}
\begin{theorem} 
The  distribution $\scrP(\alpha_1,\ldots,\alpha_n)$ with  L\'{e}vy   conditioning distributions
$\{\scrL(\alpha_1),\ldots,\scrL(\alpha_n)\}$ has marginal densities
\begin{align}
\Pr(p_i|\alpha_1,\ldots,\alpha_n) 
             &= \frac{S_{\alpha_i,\alpha-\alpha_i}(p_i)} {\pi\sqrt{p_i(1-p_i)}}   \qquad  \alpha = \alpha_1+\dots+\alpha_n, 
                                     \quad i=1\ldots n-1 \nonumber  \\
\textrm{where} \quad  S_{\alpha,\beta}(p) &= \frac{\alpha\beta}{\alpha^2 (1-p)+\beta^2 p} \quad \textrm{for any $\alpha,\beta>0$}  \nonumber
\end{align}
\label{thm:levymarginal} 
\end{theorem}
\tcblower
The denominator of $\Pr(p_i|\alpha_1,\ldots,\alpha_n)$  is the density of the beta distribution $\scrB(\tfrac{1}{2},\tfrac{1}{2})$:
\begin{align}
\frac{1}{\pi\sqrt{p(1-p)}} &\equiv \frac{\Gamma(1)}{\Gamma(\frac{1}{2})\Gamma(\frac{1}{2})}\, p^{\frac{1}{2}-1}  (1-p)^{\frac{1}{2}-1} 
\qquad \left(\Gamma(\tfrac{1}{2}) = \sqrt{\pi}\right)  \nonumber 
\end{align}
\end{tcolorbox}

\begin{proof}[Proof of Theorem \ref{thm:levymarginal}]
Let the density of  $\scrL(\alpha_i)$ be $f_i(x)\equiv f_{\alpha_i,\frac{1}{2}}(x)$ and 
let $f_{(i)}$ be the convolution of $\{f_1,\dots,f_n\}$ with $f_i$ omitted, {\it i.e.}\ $f_{(i)}(x)= f_{\alpha-\alpha_i,\frac{1}{2}}(x)$
by closure of stable distributions under convolution.
Hence, by Corollary~\ref{cor:marginal},
\begin{align}
\Pr(p_i|\alpha_1,\ldots,\alpha_n) 
      &= \int_0^\infty z f_i(zp_i) f_{(i)}(z(1-p_i)) dz \nonumber  \\
      &= \int_0^\infty z f_{\alpha_i,\frac{1}{2}}(zp_i) f_{\beta_i,\frac{1}{2}}(z(1-p_i))dz \qquad \beta_i=\alpha-\alpha_i \nonumber \\
      &= \int_0^\infty z \left[\frac{\alpha_i\,e^{-\alpha_i^2/4zp_i}}{2\sqrt{\pi (zp_i)^3}}\right]
                                 \left[\frac{\beta_i\,e^{-\beta_i^2/4z(1-p_i)}}{2\sqrt{\pi (z(1-p_i))^3}}\right] dz \nonumber \\
      &= \frac{\alpha_i \beta_i}{4\pi\sqrt{p_i^3(1-p_i)^3}} 
            \int_0^\infty   \exp-\tfrac{u}{4}\left[\tfrac{\alpha_i^2(1-p_i)+\beta_i^2p_i}{p_i(1-p_i)}\right] \; du \nonumber \\
      &= \frac{S_{\alpha_i,\beta_i}(p_i)} {\pi\sqrt{p_i(1-p_i)}} 
\label{eq:levymarginal}
\end{align}
with $S_{\alpha,\beta}(p)$ as defined in the statement of Theorem~\ref{thm:levymarginal}.
\end{proof}

The beta distribution term in~(\ref{eq:levymarginal}) has a  fixed shape parameter of $1/2$,
{\it i.e.}\ it concentrates the probability mass toward the axes as discussed earlier.
The associated factor $S_{\alpha,\beta}(p)$  on $0\le p\le 1$ has the following properties:
\vspace{-1pc}
\begin{myitemize}
\item $S_{\alpha,\beta}(p)>0$.
\item $S_{\alpha,\beta}(p)$ is monotonic  from $S_{\alpha,\beta}(0)=\beta/\alpha$ to $S_{\alpha,\beta}(1)=\alpha/\beta$.
 $S_{\alpha,\alpha}(p)=1$.
\item $S_{\alpha,\beta}(p)$ is invariant under a global scale change $(\alpha,\beta)\to(\lambda\alpha,\lambda\beta),\,\lambda>0$.
\end{myitemize}
The effect of  $S_{\alpha,\beta}(p)$  is to skew the beta term to one end or the other depending on the ratio $\alpha/\beta$.
But, for all $\alpha/\beta$ combinations, the probability mass remains  concentrated toward the axes.
This contrasts with  the Dirichlet case, where the probability mass is concentrated away from the axes for $(\alpha,\beta)>1$ 
and toward the axes for $(\alpha,\beta)<1$.

The attributes of distributions and  processes imposed by  L\'{e}vy conditioning distributions and other stable distributions 
warrant more detailed study.
We  aim to  pursue such study  in a separate paper.


\section{Conclusion}
\label{sec:conclusion}
In presenting this paper, we have followed  the thought process from an initial interest in the  growth of random trees by leaf attachment  
to a general discussion of probabilistic  attachment based on vertex clusters. 
Leaf attachment itself was inspired by distributed ledgers, which can be modelled as random trees
that grow by leaf attachment, but the thought process quickly evolved from distributed ledgers to the growth of trees in general.
Indeed, the very mention of  distributed ledgers proved to be a distraction because of the many issues that come up that are not, in their nature, about probabilistic modelling.

The benefit of starting with leaf attachment is that it led rather naturally to cluster thinking.
This together  with  intrinsic vertex masses that, in turn,  induce cluster masses through additivity led to a rich probabilistic 
framework and a novel theorem.
A special case of the theorem led to the Dirichlet distribution, whose mean gives the degree distributions of preferential attachment.

We concluded by proposing another choice of distributions consistent with the general theorem, the stable distributions such as the L\'{e}vy distribution that we propose to explore further in a separate paper dedicated to what we refer to as stable random trees 
that grow by probabilistic attachment. 

\bibliography{Leaf_2021}{}

\begin{thebibliography}{10}

\bibitem{Barabasi}
Albert-L\'{a}szl\'{o} Barab\'{a}si.
\newblock {\em Network Science}.
\newblock Cambridge University Press, 2017.
\newblock \href{https://networksciencebook.com}{networksciencebook.com}.

\bibitem{BarabasiAlbert}
Albert-L\'{a}szl\'{o} Barab\'{a}si and R\'{e}ka Albert.
\newblock Emergence of scaling in random networks.
\newblock {\em Science}, 286(5439):509--512, 1999.

\bibitem{BianconiBarabasi2}
Ginestra Bianconi and Albert-L\'{a}szl\'{o} Barab\'{a}si.
\newblock Bose-{E}instein condensation in complex networks.
\newblock {\em Physical Review Letters}, 86(24):5632–5635, 2001.

\bibitem{BianconiBarabasi1}
Ginestra Bianconi and Albert-L\'{a}szl\'{o} Barab\'{a}si.
\newblock Competition and multiscaling in evolving networks.
\newblock {\em Europhysics Letters}, 54(4):436--442, 2001.

\bibitem{Caldarelli}
G.~Caldarelli, A.~Capocci, P.~De~Los Rios, and M.~A. Munoz.
\newblock Scale-free networks from varying vertex intrinsic fitness.
\newblock {\em Physical Review Letters}, 89(25):258702, 2002.

\bibitem{CaronFox}
François Caron and Emily~B. Fox.
\newblock Sparse graphs using exchangeable random measures.
\newblock {\em Journal of the Royal Statistical Society B}, 79(5):1295--1366,
  2017.

\bibitem{Coscia}
Michele Coscia.
\newblock {\em The Atlas for the Aspiring Network Scientist}.
\newblock Michele Coscia, 2021.
\newblock \href{https://arxiv.org/pdf/2101.00863}{arXiv:2101.00863}.

\bibitem{Ferguson1}
Thomas~S. Ferguson.
\newblock A {B}ayesian analysis of some nonparametric problems.
\newblock {\em The Annals of Statistics}, 1(2):209--230, 1973.

\bibitem{Ferguson2}
Thomas~S. Ferguson.
\newblock Prior distributions on spaces of probability measures.
\newblock {\em The Annals of Statistics}, 2(4):615--629, 1974.

\bibitem{GaravagliaHofstad}
Alessandro Garavaglia, Remco~{van der} Hofstad, and Gerhard Woeginger.
\newblock The dynamics of power laws: Fitness and aging in preferential
  attachment trees.
\newblock {\em Journal of Statistical Physics}, 168(6):1137--1179, 2017.

\bibitem{vdHofstad}
Remco~{van der} Hofstad.
\newblock {\em Random Graphs and Complex Networks, Vol. I}.
\newblock Cambridge University Press, 2017.

\bibitem{Kingman}
J.~F.~C. Kingman.
\newblock Random discrete distributions.
\newblock {\em Journal of the Royal Statistical Society B}, 37(1):1--22, 1975.

\bibitem{Marchand}
David~Corlin Marchand and Ioan Manolescu.
\newblock Influence of the seed in affine preferential attachment trees.
\newblock {\em Bernoulli}, 26(3):1665--1705, 2020.

\bibitem{Newman}
Mark Newman.
\newblock {\em Networks: Second Edition}.
\newblock Oxford University Press, 2018.

\bibitem{Price}
Derek J.~{de Solla} Price.
\newblock Networks of scientific papers.
\newblock {\em Science}, 149(3683):510--515, 1965.

\bibitem{SibisiSkilling}
Sibusiso Sibisi and John Skilling.
\newblock Prior distributions on measure space.
\newblock {\em Journal of the Royal Statistical Society B}, 59(1):217--235,
  1997.

\end{thebibliography}
\bibliographystyle{plain} 

\end{document}